\documentclass{amsart}
\usepackage{amsfonts}

\usepackage{cases}
\usepackage{graphicx}
\usepackage{amssymb}
\usepackage{mathrsfs}

\newtheorem{theorem}{Theorem}[section]

\newtheorem{lemma}[theorem]{Lemma}

\numberwithin{equation}{section}

\newcommand{\bqn}{\begin{equation}}
\newcommand{\eqn}{\end{equation}}

\input{tcilatex}

\begin{document}
\title[]{Optimal evaluations for the S\'{a}ndor-Yang mean by power mean}
\author{Zhen-Hang Yang}
\address{Zhen-Hang Yang, School of Mathematics and Computation Sciences,
Hunan City University, Yiyang 413000, China}
\email{yzhkm@163.com}
\thanks{}
\author{Yu-Ming Chu}
\address{Yu-Ming Chu (Corresponding author), School of Mathematics and
Computation Sciences, Hunan City University, Yiyang 413000, China}
\email{chuyuming@hutc.zj.cn}

\begin{abstract}
In this paper, we prove that the double inequality 
\begin{equation*}
M_{p}(a,b)<B(a,b)<M_{q}(a,b)
\end{equation*}
holds for all $a, b>0$ with $a\neq b$ if and only if $p\leq 4\log 2/(4+2\log
2-\pi)=1.2351\cdots$ and $q\geq 4/3$, where $%
M_{r}(a,b)=[(a^{r}+b^{r})/2]^{1/r}$ $(r\neq 0)$ and $M_{0}(a,b)=\sqrt{ab}$
is the $r$th power mean, $B(a,b)=Q(a,b)e^{A(a,b)/T(a,b)-1}$ is the S\'{a}%
ndor-Yang mean, $A(a,b)=(a+b)/2$, $Q(a,b)=\sqrt{(a^{2}+b^{2})/2}$ and $%
T(a,b)=(a-b)/[2\arctan((a-b)/(a+b))]$.
\end{abstract}

\maketitle




\section{Introduction}

\bigskip For $r\in \mathbb{R}$, the $r$th power mean $M_{r}(a,b)$ of two
distinct positive real numbers $a$ and $b$ is defined by 
\begin{equation}
M_{r}(a,b)=%
\begin{cases}
\left(\frac{a^{r}+b^{r}}{2}\right)^{1/r}, \quad r\neq 0, \\ 
\sqrt{ab}, \quad r=0.%
\end{cases}%
\end{equation}

It is well known that $M_{r}(a,b)$ is continuous and strictly increasing
with respect to $r\in \mathbb{R}$ for fixed $a, b>0$ with $a\neq b$. Many
classical means are the special cases of the power mean, for example, $%
M_{-1}(a,b)=2ab/(a+b)=H(a,b)$ is the harmonic mean, $M_{0}(a,b)=\sqrt{ab}%
=G(a,b)$ is the geometric mean, $M_{1}(a,b)=(a+b)/2=A(a,b)$ is the
arithmetic mean, and $M_{2}(a,b)=\sqrt{(a^{2}+b^{2})/2}=Q(a,b)$ is the
quadratic mean. The main properties for the power mean are given in [1].

Let 
\begin{equation*}
L(a,b)=\frac{a-b}{\log a-\log b}, ~~I(a,b)=\frac{1}{e}\left(\frac{a^{a}}{%
b^{b}}\right)^{1/(a-b)}, ~~P(a,b)=\frac{a-b}{2\arcsin\left(\frac{a-b}{a+b}%
\right)},
\end{equation*}
\begin{equation*}
U(a,b)=\frac{a-b}{\sqrt{2}\arctan\left(\frac{a-b}{\sqrt{2ab}}\right)},
~~T^{\ast}(a,b)=\frac{2}{\pi}\int_{0}^{\pi/2}\sqrt{a^{2}\cos^{2}\theta+b^{2}%
\sin^{2}\theta}d\theta,
\end{equation*}
\begin{equation*}
NS(a,b)=\frac{a-b}{2\sinh^{-1}\left(\frac{a-b}{a+b}\right)},
~~X(a,b)=A(a,b)e^{G(a,b)/P(a,b)-1},
\end{equation*}
\begin{equation}
T(a,b)=\frac{a-b}{2\arctan\left(\frac{a-b}{a+b}\right)},
~~B(a,b)=Q(a,b)e^{A(a,b)/T(a,b)-1}
\end{equation}
be respectively the logarithmic mean, identric mean, first Seiffert mean
[2], Yang mean [3], Toader mean [4], Neuman-S\'{a}ndor mean [5, 6], S\'{a}%
ndor mean [7], second Seiffert mean [8], S\'{a}ndor-Yang mean [3] of $a$ and 
$b$.

Recently, the sharp bounds for certain bivariate means in terms of the power
mean have attracted the attention of many mathematicians. Lin [9] proved
that the double inequality 
\begin{equation*}
M_{p}(a,b)<L(a,b)<M_{q}(a,b)
\end{equation*}
holds for all $a, b>0$ with $a\neq b$ if and only if $p\leq 0$ and $q\geq 1/3
$.

Stolarsky [10] and Pittenger [11] found that $M_{2/3}(a,b)$ and $M_{\log
2}(a,b)$ are respectively the best possible lower and upper power mean
bounds for the identric mean $I(a,b)$. In [12-15], the authors proved that
the double inequality 
\begin{equation*}
M_{p}(a,b)<T^{\ast}(a,b)<M_{q}(a,b)
\end{equation*}
holds for all $a, b>0$ with $a\neq b$ if and only if $p\leq 3/2$ and $q\geq
\log 2/(\log\pi-\log 2)$.

Jagers [16], H\"{a}st\"{o} [17, 18], Yang [19], and Costin and Toader [20]
proved that $p_{1}=\log 2/\log\pi$, $q_{1}=2/3$, $p_{2}=\log 2/(\log\pi-\log
2)$ and $q_{2}=5/3$ are the best possible parameters such that the double
inequalities 
\begin{equation*}
M_{p_{1}}(a,b)<P(a,b)<M_{q_{1}}(a,b), ~~M_{p_{2}}(a,b)<T(a,b)<M_{q_{2}}(a,b)
\end{equation*}
hold for all $a, b>0$ with $a\neq b$.

In [20-25], the authors proved that the double inequalities 
\begin{equation*}
M_{\lambda_{1}}(a,b)<NS(a,b)<M_{\mu_{1}}(a,b),
\end{equation*}
\begin{equation*}
M_{\lambda_{2}}(a,b)<U(a,b)<M_{\mu_{2}}(a,b),
\end{equation*}
\begin{equation*}
M_{\lambda_{3}}(a,b)<X(a,b)<M_{\mu_{3}}(a,b)
\end{equation*}
hold for all $a, b>0$ with $a\neq b$ if and only if $\lambda_{1}\leq \log 2/%
\log[2\log(1+\sqrt{2})]$, $\mu_{1}\geq 4/3$, $\lambda_{2}\leq 2\log
2/(2\log\pi-\log 2)$, $\mu_{2}\geq 4/3$, $\lambda_{3}\leq 1/3$ and $%
\mu_{3}\geq \log 2/(1+\log 2)$.

Yang et. al. [26] proved that 
\begin{equation}
M_{1}(a,b)<B(a,b)<M_{2}(a,b)
\end{equation}
for all $a, b>0$ with $a\neq b$.

Motivated by inequality (1.3), it is natural to ask what are the greatest
value $p$ and the least value $q$ such that the double inequality 
\begin{equation*}
M_{p}(a,b)<B(a,b)<M_{q}(a,b)
\end{equation*}
holds for all $a, b>0$ with $a\neq b$? The main purpose of this paper is to
answer this question.

\bigskip

\section{Lemmas}

\bigskip

In order to prove our main results we need several lemmas, which we present
in this section.

\setcounter{equation}{0}

\bigskip

\begin{lemma}
(See [27, Lemma 7]) Let $\{a_{k}\}_{k=0}^{\infty}$ be a nonnegative real
sequence with $a_{m}>0$ and $\sum_{k=m+1}^{\infty}a_{k}>0$, and 
\begin{equation*}
P(t)=\sum_{k=0}^{m}a_{k}t^{k}-\sum_{k=m+1}^{\infty}a_{k}t^{k}
\end{equation*}
be a convergent power series on the interval $(0, \infty)$. Then there
exists $t_{m+1}\in (0, \infty)$ such that $P(t_{m+1})=0$, $P(t)>0$ for $t\in
(0, t_{m+1})$ and $P(t)<0$ for $t\in (t_{m+1}, \infty)$.
\end{lemma}

\bigskip

\begin{lemma}
(See [22, Lemma 6]) The function $r\rightarrow 2^{1/r}M_{r}(a, b)$ is
strictly decreasing and log-convex on $(0, \infty)$ for all $a, b>0$ with $%
a\neq b$.
\end{lemma}

\medskip

\begin{lemma}
Let $t>0$, $p\in \mathbb{R}$ and 
\begin{equation}
f_{1}(t, p)=-\arctan(\tanh(t))+\sinh(t)\cosh(t)-\tanh(pt)\sinh^{2}(t).
\end{equation}
Then the following statements are true:

$(i)$ if $p\leq 1$, then $f_{1}(t, p)$ is strictly increasing with respect
to $t$ on $(0, \infty)$;

$(ii)$ if $p\geq 4/3$, then $f_{1}(t, p)$ is strictly decreasing with
respect to $t$ on $(0, \infty)$;

$(iii)$ if $p\in (1, 4/3)$, then there exists $t_{1}\in (0, \infty)$ such
that $f_{1}(t, p)$ is strictly increasing with respect to $t$ on $(0, t_{1})$
and strictly decreasing with respect to $t$ on $(t_{1}, \infty)$.
\end{lemma}

\begin{proof}
Let 
\begin{equation}
u_{n}(p)=(2-p)^{2n}-p^{2n}+(1-p)2^{2n}+2p,
\end{equation}
\begin{equation*}
f_{2}(t, p)=4\sinh^{2}(t)\cosh^{2}\left(\frac{pt}{2}\right)-4p\cosh(t)%
\sinh^{2}\left(\frac{t}{2}\right)-\sinh(2t)\sinh(pt).
\end{equation*}
Then simple computations lead to 
\begin{equation}
u_{1}\left(\frac{4}{3}\right)=0, ~~u_{n}\left(\frac{4}{3}\right)=-\frac{%
4^{2n}-2^{2n}}{3^{2n}}-\frac{2^{2n}-8}{3}<0 ~~(n\geq 2),
\end{equation}
\begin{equation}
\frac{\partial f_{1}(t, p)}{\partial t}=-\frac{1}{\cosh(2t)}+\cosh(2t)-\frac{%
p\sinh^{2}(t)}{\cosh^{2}(pt)}-\tanh(pt)\sinh(2t)
\end{equation}
\begin{equation*}
=\frac{f_{2}(2t, p)}{4\cosh(2t)\cosh^{2}(pt)},
\end{equation*}
\begin{equation}
f_{2}(t,p)=\cosh[(p-2)t]-\cosh(pt)+(1-p)\cosh(2t)+2p\cosh(t)-p-1
\end{equation}
\begin{equation*}
=\sum_{n=1}^{\infty}\frac{u_{n}(p)}{(2n)!}t^{2n},
\end{equation*}
\begin{equation}
\frac{\partial f_{2}(t,p)}{\partial p}=2\cosh(t)-\cosh(2t)+t\sinh[(p-2)t]%
-t\sinh(pt)-1
\end{equation}
\begin{equation*}
=-2[\cosh(t)-1]\cosh(t)-2t\sinh(t)\cosh[(p-1)t]<0.
\end{equation*}

$(i)$ If $p\leq 1$, then equations (2.5) and (2.6) lead to 
\begin{equation}
f_{2}(t,p)\geq f_{2}(t, 1)=2[\cosh(t)-1]>0.
\end{equation}

Therefore, Lemma 2.3$(i)$ follows easily from (2.4) and (2.7).

$(ii)$ If $p\geq 4/3$, then from (2.3), (2.5) and (2.6) we have 
\begin{equation}
f_{2}(t,p)\leq f_{2}\left(t, \frac{4}{3}\right)=\sum_{n=1}^{\infty}\frac{%
u_{n}(4/3)}{(2n)!}t^{2n}<0.
\end{equation}

Therefore, Lemma 2.3$(ii)$ follows easily from (2.4) and (2.8).

$(iii)$ If $p\in (1, 4/3)$, then from (2.4) it is enough to prove that there
exists $t_{1}\in (0, \infty)$ such that $f_{2}(t, p)>0$ for $t\in (0, t_{1})$
and $f_{2}(t, p)<0$ for $t\in (t_{1}, \infty)$.

It follows from (2.2) that 
\begin{equation}
u_{1}(p)=2(4-3p)>0, ~~\lim_{n\rightarrow \infty}\frac{u_{n}(p)}{2^{2n}}%
=1-p<0,
\end{equation}
\begin{equation}
u_{n+1}(p)-u_{n}(p)=-(p-1)\left[(3-p)(2-p)^{2n}+3\times 2^{2n}+(p+1)p^{2n}%
\right]<0
\end{equation}
for all $n\geq 1$.

Therefore, the desired result follows from (2.5), (2.9), (2.10) and Lemma
2.1.
\end{proof}

\medskip

\begin{lemma}
Let $t>0$, $p\in \mathbb{R}$ $f_{1}(t,p)$ be defined by (2.1). Then

$(i)$ $f_{1}(t, p)>0$ for all $t\in (0, \infty)$ if and only if $p\leq 1$;

$(ii)$ $f_{1}(t, p)<0$ for all $t\in (0, \infty)$ if and only if $p\geq 4/3$;

$(iii)$ there exists $t_{0}\in (0, \infty)$ such that $f_{1}(t_{0}, p)=0$, $%
f_{1}(t, p)>0$ for $t\in (0, t_{0})$ and $f_{1}(t, p)<0$ for $t\in (t_{0},
\infty)$ if $p\in (1, 4/3)$.
\end{lemma}

\begin{proof}
$(i)$ If $p\leq 1$, then Lemma 2.3$(i)$ and (2.1) lead to the conclusion
that $f_{1}(t, p)>f_{1}(0, p)=0$ for all $t\in (0, \infty)$.

If $f_{1}(t, p)>0$ for all $t\in (0, \infty)$, then $\lim_{t\rightarrow
\infty}f_{1}(t, p)\geq 0$. We claim that $p\leq 1$. Indeed, if $p>1$, then
from (2.1) we have 
\begin{equation*}
\lim_{t\rightarrow \infty}f_{1}(t, p)=\lim_{t\rightarrow \infty}\left[%
-\arctan(\tanh(t))+\frac{\sinh(t)\cosh((p-1)t)}{\cosh(pt)}\right]
\end{equation*}
\begin{equation*}
=\lim_{t\rightarrow \infty}\left[-\arctan(\tanh(t))+\frac{1-e^{-2t}}{2}\frac{%
1+e^{-2|p-1|t}}{1+e^{-2|p|t}}e^{(1+|p-1|-|p|)t}\right]
\end{equation*}
\begin{equation*}
=-\frac{\pi}{4}+\frac{1}{2}<0.
\end{equation*}

$(ii)$ If $p\geq 4/3$, then Lemma 2.3$(ii)$ and (2.1) imply that $f_{1}(t,
p)<f_{1}(0, p)=0$ for all $t\in (0, \infty)$.

If $f_{1}(t, p)<0$ for all $t\in (0, \infty)$, then we clearly see that 
\begin{equation}
\lim_{t\rightarrow 0}\frac{f_{1}(t, p)}{t^{3}}\leq 0.
\end{equation}

It follows from (2.1), (2.2), (2.4) and (2.5) that 
\begin{equation}
\lim_{t\rightarrow 0}\frac{f_{1}(t, p)}{t^{3}}=\lim_{t\rightarrow 0}\frac{%
\partial f_{1}(t, p)/\partial t}{3t^{2}} =\lim_{t\rightarrow 0}\frac{1}{%
3\cosh(2t)\cosh^{2}(pt)}\times\lim_{t\rightarrow 0}\frac{f_{2}(2t, p)}{%
(2t)^{2}}
\end{equation}
\begin{equation*}
=\frac{1}{3}\times\frac{1}{2}u_{1}(p)=-\left(p-\frac{4}{3}\right).
\end{equation*}

Inequality (2.11) and equation (2.12) lead to the conclusion that $p\geq 4/3$%
.

$(ii)$ If $p\in (1, 4/3)$, then from Lemma 2.3$(iii)$ and the facts that $%
f_{1}(0, p)=0$ and $\lim_{t\rightarrow \infty}f_{1}(t, p)=-\pi/4+1/2<0$ we
clearly see that there exists $t_{0}\in (0, \infty)$ such that $f_{1}(t_{0},
p)=0$, $f_{1}(t, p)>0$ for $t\in (0, t_{0})$ and $f_{1}(t, p)<0$ for $t\in
(t_{0}, \infty)$.
\end{proof}

\medskip

\begin{lemma}
Let $t>0$, $p\in (-\infty, 0)\cup (0, \infty)$ and 
\begin{equation}
F(t, p)=\frac{1}{2}\log[\cosh(2t)]+\frac{\arctan(\tanh(t))}{\tanh(t)}-\frac{1%
}{p}\log[\cosh(pt)]-1.
\end{equation}
Then

$(i)$ $F(t, p)$ is strictly increasing with respect to $t$ on $(0, \infty)$
if and only if $p\leq 1$;

$(ii)$ $F(t, p)$ is strictly decreasing with respect to $t$ on $(0, \infty)$
if and only if $p\geq 4/3$;

$(iii)$ there exists $t_{0}\in (0, \infty)$ such that $f_{1}(t_{0}, p)=0$, $%
F(t, p)$ is strictly increasing with respect to $t$ on $(0, t_{0})$ and
strictly decreasing with respect to $t$ on $(t_{0}, \infty)$, where $%
f_{1}(t, p)$ is defined by (2.1).
\end{lemma}

\begin{proof}
It follows from (2.13) that 
\begin{equation}
\frac{\partial F(t, p)}{\partial t}=\frac{-\arctan(\tanh(t))+\sinh(t)%
\cosh(t)-\tanh(pt)\sinh^{2}(t)}{\sinh^{2}(t)}=\frac{f_{1}(t, p)}{\sinh^{2}(t)%
}.
\end{equation}

Therefore, Lemma 2.5 follows from Lemma 2.4 and (2.14).
\end{proof}

\bigskip

\section{Main Results}

\bigskip

\medskip

\begin{theorem}
The inequality 
\begin{equation}
B(a,b)<M_{p}(a,b)
\end{equation}
holds for all $a, b>0$ with $a\neq b$ if and only if $p\geq 4/3$. Moreover,
the inequality 
\begin{equation}
B(a,b)>\lambda_{p}M_{p}(a,b)
\end{equation}
holds for all $a, b>0$ and $a\neq b$ with the best possible parameter $%
\lambda_{p}=e^{\pi/4-1}2^{1/p-1/2}$ if $p\geq 4/3$.
\end{theorem}

\begin{proof}
Since $B(a,b)$ and $M(a,b)$ are symmetric and homogeneous of degree 1,
without loss of generality, we assume that $b>a>0$. Let $t=\log\sqrt{b/a}>0$%
, $p\in \mathbb{R}$ and $p\neq 0$, $f_{1}(t, p)$ and $F(t, p)$ be defined by
(2.1) and (2.13), respectively. Then (1.1), (1.2), (2.1), (2.12), (2.13) and
(2.14) lead to 
\begin{equation*}
M_{p}(a,b)=\sqrt{ab}\cosh^{1/p}(pt), ~~T(a,b)=\sqrt{ab}\frac{\sinh(t)}{%
\arctan[\tanh(t)]},
\end{equation*}
\begin{equation*}
B(a,b)=\sqrt{ab}\cosh^{1/2}(2t)e^{\arctan(\tanh(t))/\tanh(t)-1},
\end{equation*}
\begin{equation}
\log[B(a,b)]-\log[M_{p}(a,b)]=F(t, p),
\end{equation}
\begin{equation}
F(0^{+}, p)=0,
\end{equation}
\begin{equation}
\lim_{t\rightarrow 0^{+}}\frac{F(t, p)}{t^{2}}=\lim_{t\rightarrow 0^{+}}%
\frac{\partial F(t, p)/\partial t}{2t} =\lim_{t\rightarrow 0^{+}}\frac{%
f_{1}(t,p)}{2t\sinh^{2}(t)}=-\frac{1}{2}\left(p-\frac{4}{3}\right),
\end{equation}
\begin{equation}
\lim_{t\rightarrow \infty}F(t,p)
\end{equation}
\begin{equation*}
=\lim_{t\rightarrow \infty}\left[\left(1-\frac{|p|}{p}\right)t+\frac{1}{2}%
\log\left(\frac{1+e^{-4t}}{2}\right) +\frac{\arctan(\tanh(t))}{\tanh(t)}-%
\frac{1}{p}\log\left(\frac{1+e^{-2|p|t}}{2}\right)-1\right]
\end{equation*}
\begin{equation*}
=\frac{1}{4}\pi-\frac{1}{2}\log 2+\frac{1}{p}\log 2-1=\log (\lambda_{p})
~~(p>0).
\end{equation*}

If $B(a,b)<M_{p}(a,b)$, then (3.3) and (3.5) lead to $p\geq 4/3$.

If $p\geq 4/3$, then from (3.4) and (3.6) together Lemma 2.5$(ii)$ we
clearly see that 
\begin{equation}
\log (\lambda_{p})<\lim_{t\rightarrow \infty}F(t, p)<F(t,p)<F(0^{+}, p)=0
\end{equation}
for all $t>0$ with the best possible parameter $\lambda_{p}$.

Therefore, the double inequality 
\begin{equation*}
\lambda_{p}M_{p}(a,b)<B(a,b)<M_{p}(a,b)
\end{equation*}
holds for all $a, b>0$ and $a\neq b$ with the best possible parameter $%
\lambda_{p}$ follows from (3.3) and (3.7).
\end{proof}

\medskip

Note that 
\begin{equation}
\lambda_{p}M_{p}(a,b)=\frac{\sqrt{2}}{2}e^{\pi/4-1}\left(2^{1/p}M_{p}(a,b)%
\right), ~~\lim_{p\rightarrow \infty}M_{p}(a,b)=\max\{a, b\}.
\end{equation}

Let $p=4/3$, $3/2$, $2, 3, \cdots, \infty$. Then from Lemma 2.2, (3.1),
(3.2) and (3.8) together with the monotonicity of the function $p\rightarrow
M_{p}(a,b)$ we get Corollary 3.1. \medskip

\noindent\textbf{Corollary 3.1}. The inequalities 
\begin{equation*}
\lambda_{\infty}\max\{a, b\}<\cdots<\lambda_{2}M_{2}(a,b)<
\lambda_{3/2}M_{3/2}(a,b)<\lambda_{4/3}M_{4/3}(a,b)
\end{equation*}
\begin{equation*}
<B(a,b)<M_{4/3}(a,b)<M_{3/2}(a,b)<M_{2}(a,b)<\cdots<\max\{a, b\}
\end{equation*}
hold for all $a, b>0$ and $a\neq b$ with the best possible parameters $%
\lambda_{\infty}=\frac{\sqrt{2}}{2}e^{\pi/4-1}=0.5705\cdots$, $%
\lambda_{2}=e^{\pi/4-1}=0.8068\cdots$, $\lambda_{3/2}=2^{1/6}e^{%
\pi/4-1}=0.9056\cdots$ and $\lambda_{4/3}=2^{1/4}e^{\pi/4-1}=0.9595\cdots$.
\bigskip

\begin{theorem}
Let $p_{0}=4\log 2/(4+2\log 2-\pi)=1.2351\cdots$. Then the inequality 
\begin{equation}
B(a,b)>M_{p}(a,b)
\end{equation}
holds for all $a, b>0$ with $a\neq b$ if and only if $p\leq p_{0}$.
\end{theorem}

\begin{proof}
If $B(a,b)>M_{p}(a,b)$, then (3.3) and (3.6) lead to $p\leq p_{0}$.

If $p=p_{0}$, then (3.4), (3.6) and Lemma 2.5$(iii)$ lead to the conclusion
that 
\begin{equation}
F(0^{+}, p_{0})=\lim_{t\rightarrow \infty}F(t, p_{0})=0
\end{equation}
and there exists $t_{0}\in (0, \infty)$ such that the function $t\rightarrow
F(t, p_{0})$ is strictly increasing on $(0, t_{0})$ and strictly decreasing
on $(t_{0}, \infty)$.

Therefore, 
\begin{equation*}
B(a,b)>M_{p_{0}}(a,b)>M_{p}(a,b)
\end{equation*}
for all $p\leq p_{0}$ follows easily from (3.3), (3.10), the piecewise
monotonicity of the function $t\rightarrow F(t, p_{0})$ and the monotonicity
of the function $p\rightarrow M_{p}(a,b)$.
\end{proof}

\bigskip

\noindent\textbf{Corollary 3.2.} Let $f_{1}(t, p)$, $F(t, p)$ and $%
\lambda_{p}$ be defined respectively by (2.1), (2.13) and Theorem 3.1, and $%
p_{0}=4\log 2/(4+2\log 2-\pi)=1.2351\cdots$. Then the inequality 
\begin{equation}
B(a,b)<\lambda_{p}M_{p}(a,b)
\end{equation}
holds for all $a, b>0$ and $a\neq b$ with the best possible parameter $%
\lambda_{p}$ if $p\in (0, 1]$, and the inequality 
\begin{equation}
B(a,b)\leq e^{F(t_{0}, p)}M_{p}(a,b)
\end{equation}
holds for all $a, b>0$ and $a\neq b$ with the best possible parameter $%
e^{F(t_{0}, p)}$ if $p\in (1, p_{0}]$, where $t_{0}$ is the unique solution
of the equation $f_{1}(t, p)=0$ on the interval $(0, \infty)$. In
particular, Numerical computations show that $e^{F(t_{0}, p_{0})}=1.012\cdots
$.

\begin{proof}
If $p\in (0, 1]$, then inequality (3.11) holds for all $a, b>0$ and $a\neq b$
with the best possible parameter $\lambda_{p}$ follows from (3.3) and (3.6)
together with Lemma 2.5$(i)$.

If $p\in (1, p_{0}]$, then inequality (3.12) holds for all $a, b>0$ and $%
a\neq b$ with the best possible parameter $e^{F(t_{0}, p)}$ follows from
(3.3) and Lemma 2.5$(iii)$.
\end{proof}

\medskip

Let $p\in \mathbb{R}$, $b>a>0$, $L_{p}(a,b)=\left(a^{p+1}+b^{p+1}\right)/%
\left(a^{p}+b^{p}\right)$ be the $p$th Lehmer mean [28] of $a$ and $b$, $%
f_{1}(t, p)$ be defined by (2.1), and $t=\log\sqrt{b/a}>0$. Then $f_{1}(t,p)$
can be rewritten as 
\begin{equation}
f_{1}(t,p)=-\arctan(\tanh(t))+\sinh(t)\frac{\cosh((p-1)t)}{\cosh(pt)}
\end{equation}
\begin{equation*}
=\frac{\arctan(\tanh(t))\cosh((p-1)t)}{\cosh(pt)}\left(\frac{\sinh(t)}{%
\arctan(\tanh(t))}-\frac{\cosh(pt)}{\cosh((p-1)t)}\right)
\end{equation*}

\begin{equation*}
=\frac{\arctan(\tanh(t))\cosh((p-1)t)}{\sqrt{ab}\cosh(pt)}%
\left(T(a,b)-L_{p-1}(a,b)\right).
\end{equation*}
\medskip

Lemma 2.4 and (3.13) lead to Corollary 3.3 immediately. \medskip

\noindent\textbf{Corollary 3.3.} (see [29, Theorem 2.2]) The double
inequality 
\begin{equation*}
L_{p}(a,b)<T(a,b)<L_{q}(a,b)
\end{equation*}
holds for all $a, b>0$ with $a\neq b$ if and only if $p\leq 0$ and $q\geq 1/3
$. \medskip

\noindent\textbf{Corollary 3.4.} The double inequality 
\begin{equation*}
\lambda L_{1/3}(a,b)<T(a,b)<\mu L_{0}(a,b)
\end{equation*}
holds for all $a, b>0$ with $a\neq b$ if and only if $\lambda\leq 2/\pi$ and 
$\mu\geq 4/\pi$.

\begin{proof}
Without loss of generality, we assume that $b>a>0$. Let $t=\log\sqrt{b/a}>0$%
. Then simple computations lead to 
\begin{equation}
\frac{T(a,b)}{L_{\frac{1}{3}}(a,b)}=\frac{\sinh(t)\cosh\left(\frac{t}{3}%
\right)}{\cosh\left(\frac{4t}{3}\right)\arctan(\tanh(t))}, ~~\frac{T(a,b)}{%
L_{0}(a,b)}=\frac{\sinh(t)}{\cosh(t)\arctan(\tanh(t))},
\end{equation}
\begin{equation}
\lim_{t\rightarrow \infty}\frac{\sinh(t)\cosh\left(\frac{t}{3}\right)}{%
\cosh\left(\frac{4t}{3}\right)\arctan(\tanh(t))}=\frac{2}{\pi},
~~\lim_{t\rightarrow \infty}\frac{\sinh(t)}{\cosh(t)\arctan(\tanh(t))}=\frac{%
4}{\pi}.
\end{equation}

The log-convexity of the function $r\rightarrow 2^{1/r}M_{r}(a,b)$ given by
Lemma 2.2 implies that 
\begin{equation*}
\left(2^{3/5}M_{5/3}(a,b)\right)^{3/4}\left(2^{3}M_{1/3}(a,b)%
\right)^{1/4}>2^{3/4}M_{4/3}(a,b),
\end{equation*}
which can be rewritten as 
\begin{equation}
\frac{2^{8/5}}{\pi}M_{5/3}(a,b)>\frac{2}{\pi}\frac{M_{4/3}^{4/3}(a,b)}{%
M_{1/3}^{1/3}(a,b)}=\frac{2}{\pi}L_{1/3}(a,b).
\end{equation}

Yang et. al. [30] and Witkowski [31] proved that 
\begin{equation}
\frac{2^{8/5}}{\pi}M_{5/3}(a,b)<T(a,b)<\frac{4}{\pi}A(a,b)=\frac{4}{\pi}%
L_{0}(a,b).
\end{equation}

Therefore, Corollary 3.4 follows from (3.14)-(3.17).
\end{proof}

\medskip \noindent\textbf{Competing interests}

\noindent{The authors declare that they have no competing interests.}

\medskip \noindent\textbf{Authors' contributions}

\noindent{All authors contributed equally to the writing of this paper. All
authors read and approved the final manuscript.}

\medskip \noindent\textbf{Acknowledgements}

\noindent{The research was supported by the Natural Science Foundation of
China under Grants 61374086 and 11171307, and the Natural Science Foundation
of Zhejiang Province under Grant LY13A010004.} 

\end{document}